\documentclass[a4paper,reqno]{elsarticle}
\usepackage{newcent} 
\usepackage[T1]{fontenc}
\usepackage[centertags]{amsmath}
\usepackage{amsfonts}
\usepackage{amssymb}
\usepackage{amsthm}
\usepackage{newlfont}
\usepackage{fancyhdr,fancyvrb}
\usepackage{graphicx}
\usepackage{pinlabel}
\usepackage{nextpage}
\usepackage[dvipsnames,usenames]{color}
\usepackage{hyperref}
\usepackage{esint}
\usepackage{multirow}
\usepackage{fractalfractures-private}
\usepackage{marginnote}

\numberwithin{equation}{section}

\newtheorem{defin}{Definition}[section]
\newtheorem{theorem}[defin]{Theorem}
\newtheorem{lemma}[defin]{Lemma}
\newtheorem{proposition}[defin]{Proposition}

\theoremstyle{definition} {\newtheorem{remark}[defin]{Remark}}

\date{\today}

\makeindex

\begin{document}
\title{A model for the quasistatic growth of cracks with fractional dimension\tnoteref{t1}\tnoteref{t2}}%
\author[SISSA]{Gianni Dal Maso}
\ead{dalmaso@sissa.it}
\author[SISSA]{Marco Morandotti\corref{cor1}}
\ead{morandot@sissa.it}
\address[SISSA]{SISSA, Via Bonomea, 265, 34136 Trieste, Italy}
\cortext[cor1]{Corresponding author}
\tnotetext[t1]{This paper is dedicated to Nicola Fusco on the occasion of his $60^\mathrm{th}$ birthday.}
\tnotetext[t2]{Preprint SISSA 10/2016/MATE.}

\begin{abstract}
We study a variational model for the quasistatic growth of cracks with fractional dimension in brittle materials.
We give a minimal set of properties of the collection of admissible cracks which ensure the existence of a quasistatic evolution.
Both the antiplane and the planar cases are treated.
\end{abstract}
\maketitle


\smallskip
\noindent {\bf {2010} MSC:} Primary  
49J45, 
Secondary 
74R10, 
49J40, 
28A78. 

\tableofcontents

\section{Introduction}
The current models of linearly elastic fracture mechanics are based on Griffith's energy criterion \cite{Griffith}.
Crack growth is the competition between the elastic energy released when the crack grows and the energy spent to produce new crack.
In the standard case of planar elasticity it is tacitly assumed that cracks open along one-dimensional sets $K$ in the reference configuration $\Omega\subset\R2$, so that it is natural to assume that the energy spent to produce a crack $K$ is proportional to its one-dimensional Hausdorff measure $\cH^1(K)$.
The elastic energy stored in the uncracked region is given by
\begin{equation*}
\cE_\C(u,K):=\frac12\int_{\Omega\setminus K} \C(x)Eu(x){:}Eu(x)\,\de x,
\end{equation*}
where $\C(x)$ is the elasticity tensor, $Eu$ is the symmetrized gradient of the displacement $u\colon\Omega\setminus K\to\R2$, and $:$ denotes the scalar product of $2{\times}2$ matrices.

In \cite{FrancfortMarigo1998} Francfort and Marigo introduced a quasistatic variational model based on these ideas (we refer the reader to \cite{BFM} for a general review on the variational approach to fracture mechanics).
A time-dependent boundary datum $w(t,x)$ is prescribed on a sufficiently regular portion $\partial_D\Omega$ of the boundary $\partial\Omega$. 
In the discrete-time formulation, for every integer $n>1$ one divides the time interval $[0,T]$ by $n+1$ subdivision points $0=t_n^0<t_n^1<\cdots<t_n^n=T$.
Assuming for simplicity that the proportionality constant in the crack energy is equal to $1$, the solution $(u_n^i,K_n^i)$ at time $t_n^i$ is obtained inductively by minimizing the functional
\begin{equation}\label{i002}
\cE_\C^\tot(u,K):=\cE_\C(u,K)+\cH^1(K),
\end{equation}
among all pairs $(u,K)$ such that $K\supset K_n^{i-1}$, $u$ is sufficiently smooth in $\Omega\setminus K$, and $u=w(t_n^i)$ on $\partial_D\Omega\setminus K$.

To implement this minimization scheme it is important to fix the class $\cK$ of admissible cracks $K$ that we consider and to make precise the notion of smoothness for $u$ on $\Omega\setminus K$ (see, \emph{e.g.}, Definitions \ref{222} and \ref{903}).
The existence of solutions to these problems has been obtained in \cite{DalMasoToader2002} for the antiplane case and in \cite{Chambolle2003} for the planar case, assuming that $\cK$ is the set of all compact subsets of $\cl\Omega$ with an \emph{a priori} uniform bound on the number of connected components.
This assumption is crucial to obtain the lower semicontinuity of $\cH^1(K)$ with respect to the Hausdorff distance (see Definition \ref{300}).

In many brittle materials the assumption that $\cH^1(K)<+\infty$, though useful from the mathematical point of view, is not physically justified, see \cite{Bazant,BK,Borodich1,Borodich2,Ponson,Yavari}.

Our aim in this paper is to extend the results in \cite{Chambolle2003,DalMasoToader2002} to the case where the collection $\cK$ of admissible cracks is composed of $\alpha$-dimensional sets, for some $\alpha\in(1,2)$.
We shall see that in this case we cannot take as $\cK$ the collection of all compact subsets of $\cl\Omega$ with a uniform bound on the number of connected components.
The purpose of this paper is to determine which further properties of $\cK$ are needed to obtain the results for $\alpha$-dimensional cracks.

In the discrete-time formulation, the minimum problem \eqref{i002} is replaced by the minimum problem for
\begin{equation}\label{i003}
\cE_\C^\tot(u,K):=\cE_\C(u,K)+\cH^\alpha(K)
\end{equation}
among all pairs $(u,K)$ such that $K\in\cK$, $K\supset K_n^{i-1}$, $u$ is sufficiently smooth in $\Omega\setminus K$, and $u=w(t_n^i)$ on $\partial_D\Omega\setminus K$.

The first difficulty is that $\cH^\alpha$ is not lower semicontinuous with respect to the Hausdorff metric for $\alpha\in(1,2)$, even though $\cK$ is the collection of all the connected compact subsets os $\cl\Omega$.
An immediate counterexample is given by the approximation of an $\alpha$-dimensional connected fractal set $K$ by polygonal pre-fractals $K_n$: in fact, $\cH^\alpha(K_n)=0$ for all $n$, while $\cH^\alpha(K)>0$.

To solve the minimum problem for the energy \eqref{i003}, we assume that the collection $\cK$ of admissible cracks satisfies the following properties:
\begin{enumerate}
\item[(1)] (compactness) $\cK$ is compact for the Hausdorff metric;
\item[(2)] (lower semicontinuity) $\cH^\alpha$ is lower semicontinuous on $\cK$ with respect to the Hausdorff metric.
\end{enumerate}
These properties allow us to tackle the minimum problem \eqref{i003} by the direct method of the calculus of variations (see Proposition \ref{511}).

To prove the existence of a quasistatic evolution based on the incremental minimum problems for \eqref{i003}, we introduce the piecewise constant interpolants $(u_n(t),K_n(t))$ defined by
\begin{equation*}
u_n(t):=u_n^i,\qquad K_n(t):=K_n^i,\qquad\text{for $t\in[t_n^i,t_n^{i+1})$.}
\end{equation*}
As in \cite{DalMasoToader2002} we can apply a version of Helly's Theorem and obtain, for a subsequence independent of $t$ and not relabeled, that $K_n(t)$ converges in the Hausdorff metric to a set $K(t)\in\cK$.
Moreover, there exists a function $u(t)$ such that, for a suitable subsequence (possibly depending on $t$), the functions $Eu_n(t)$ converge to $Eu(t)$ weakly in $L^2$.
As in \cite{Chambolle2003,DalMasoToader2002}, the following property is crucial for the proof of the strong convergence and of the fact that $u(t)$ minimizes $\cE_\C(u,K(t))$ among all sufficiently smooth functions $u$ on $\Omega\setminus K(t)$ such that $u=w(t)$ on $\partial_D\Omega\setminus K(t)$:
\begin{enumerate}
\item[(3)] (connectedness property) there exists $m\geq1$ such that every element of $\cK$ has at most $m$ connected components.
\end{enumerate}
Indeed, this property is the most important hypothesis to obtain the approximation result contained in Lemma \ref{nd100}.

Besides the minimality of $u(t)$ for $\cE_\C(\cdot,K(t))$, we want to prove also that the following \emph{stability condition} holds:
\begin{equation*}
\cE_\C^\tot(u(t),K(t))\leq\cE_\C^\tot(\hat u,\hat K),
\end{equation*}
for every $\hat K\in\cK$ with $\hat K\supset K(t)$ and for every sufficiently smooth function $\hat u$ on $\Omega\setminus \hat K$ such that $\hat u=w(t)$ on $\partial_D\Omega\setminus \hat K$.
To prove this inequality we need an additional property of $\cK$:
\begin{enumerate}
\item[(4)] (extension property) given $H_n,H,K\in\cK$ such that $H_n\to H\subset K$ in the Hausdorff metric, then there exists a sequence $K_n\in\cK$ such that $K_n\to K$ in the Hausdorff metric, $K_n\supset H_n$ for all $n$, and
	\begin{equation*}
	\limsup_{n\to\infty} \cH^\alpha(K_n\setminus H_n)\leq\cH^\alpha(K\setminus H).
	\end{equation*}
\end{enumerate}

Assuming (1)-(4), our main result (Theorem \ref{228b}) is the existence, under suitable assumptions on the initial and boundary conditions, of a quasistatic evolution $t\mapsto(u(t),K(t))$ satisfying the following properties: 
\begin{itemize}
\item[(a)] (irreversibility) $K(s)\subset K(t)$ for $0\leq s\leq t\leq T$;
\item[(b)] (global stability) for $0\leq t\leq T$ we have $K(t)\in\cK$, $u(t)$ is sufficiently smooth on $\Omega\setminus K(t)$, $u(t)=w(t)$ on $\partial_D\Omega\setminus K(t)$, and $\cE_\C^{\tot}(u(t),K(t))\leq\cE_\C^{\tot}(\hat u,\hat K)$ for all $\hat K\in\cK$ with $\hat K\supset K(t)$ and for all $\hat u$ sufficiently smooth on $\Omega\setminus \hat K$ and with $\hat u=w(t)$ on $\partial_D\Omega\setminus \hat K$;
\item[(c)] (energy balance) $t\mapsto\cE_\C^{\tot}(u(t),K(t))$ is absolutely continuous on $[0,T]$, and
\begin{equation*}
\frac{\de}{\de t}\cE_\C^{\tot}(w(t),K(t))=\int_{\Omega\setminus K(t)} \C(x)Eu(t,x){:}E\dot w(t,x)\,\de x \quad \text{ for a.e. } t\in[0,T],
\end{equation*}
where $\dot w$ is the partial derivative of $w$ with respect to time.
\end{itemize}

This notion of quasistatic evolution falls within the general variational framework for rate-independent systems developed in \cite{MielkeRoubicek}.
It is clear that the solution strongly depends on the choice of the collection $\cK$, which encodes the (possibly fractal) geometry of all possible cracks for the brittle material considered.

Similar problems where all sets $K\in\cK$ are contained in a single compact set $\Sigma\subset\cl\Omega$ with $\cH^\alpha(\Sigma)<+\infty$ have been studied in \cite{RaccaToader2014} with a more direct approach.
They can easily provide examples of collections $\cK$ satisfying (1)-(4).
Section \ref{sect3} of the present paper is devoted to constructing an example of a collection $\cK$ satisfying (1)-(4), which is also invariant under rotations.
The difficulties encountered in this construction show that, unlike for $\alpha=1$, it seems that there is no ``canonical'' collection $\cK$ of $\alpha$-dimensional compact sets when $\alpha\in(1,2)$.
The actual choice of $\cK$ in the fracture model will depend on the physical properties of the brittle material.

In the first part of the paper we present in detail the results for the antiplane case, where the displacement $u$ is scalar-valued and the energy depends on the gradient $\nabla u$ of $u$.
In this case, the results are proved for a more general class of energies which are not necessarily quadratic (see Definition \ref{302}).
In Section \ref{sectlinel} we outline the main differences in the case of planar elasticity, which can be treated using the approximation result proved in \cite{Chambolle2003}.

\section{Preliminaries and main result in the scalar case}
Let $\Omega\subset\R2$ be a bounded connected open set with Lipschitz boundary $\partial\Omega$.
Let $\partial_N\Omega$ be a fixed (possibly empty) closed subset of $\partial\Omega$ with a finite number of connected components, and let $\partial_D\Omega:=\partial\Omega\setminus\partial_N\Omega$.
Notice that $\partial_D\Omega$ is a relatively open set with a finite number of connected components.
\begin{defin}\label{300}
The \emph{Hausdorff distance} between two non-empty compact sets $K_1,K_2\subset\cl\Omega$ is defined by
\begin{equation*}
d_\cH(K_1,K_2):=\max\Big\{\max_{x\in K_1}\dist{x}{K_2},\max_{x\in K_2}\dist{x}{K_1}\Big\}.
\end{equation*}
Given a sequence $K_n$ of non-empty compact subsets of $\cl\Omega$ and a non-empty compact set $K\subset\cl\Omega$, we say that $K_n$ converges to $K$ in the Hausdorff metric if $d_\cH(K_n,K)\to0$.
\end{defin}
Here and in the sequel we consider a summability exponent $1<p\leq2$.
\begin{defin}[Deny-Lions space]\label{222}
Given an open subset $U$ of $\Omega$, the Deny-Lions space is defined by
\begin{equation*}
L^{1,p}(U):=\{u\in L^p_\loc(U): \nabla u\in L^p(U;\R2)\},
\end{equation*}
where $\nabla u$ denotes the gradient of $u$ in the sense of distributions.
\end{defin}
It is known that $L^{1,p}(U)$ coincides with the usual Sobolev space $W^{1,p}(U)$ if $\partial U$ is Lipschitz \cite[Sect.\@ 1.1.11]{Mazya2011}.
More precisely, if $\partial U$ is Lipschitz near a point $x\in\partial U$ and $U$ lies locally on one side of $\partial U$, then there exists an open neighborhood $V$ of $x$ such that $u\in W^{1,p}(U\cap V)$.
This allows us to define the trace of a function $u\in L^{1,p}(U)$ on the locally Lipschitz part of the boundary $\partial U$.

We now introduce the (possibly nonquadratic) energy in the scalar case.
Let $f\colon\Omega{\times}\R2\allowbreak\to[0,+\infty)$ be a function with the following properties:
\begin{subequations}\label{233}
\begin{eqnarray}
&& \text{for every $\xi\in\R2$, $f(\cdot, \xi)$ is measurable in $\Omega$;} \label{233a} \\
&& \text{for a.e.\@ $x\in\Omega$, $f(x,\cdot)$ is of class $C^1$ and strictly convex in $\R2$.} \label{233b}
\end{eqnarray}
We assume that there exist constants $0<c \leq C<+\infty$ and functions $a\in L^1(\Omega)$ and $b\in L^{q}(\Omega)$, $1/p+1/q=1$, such that
\begin{equation}\label{233c}
c|\xi|^p-a(x)\leq f(x,\xi)\leq C|\xi|^p+a(x)\quad \text{for every $x\in\Omega$, $\xi\in\R2$}.
\end{equation}
\begin{equation}\label{233d}
|\partial_\xi f(x,\xi)|\leq C|\xi|^{p-1}+b(x)\quad \text{for every $x\in\Omega$, $\xi\in\R2$},
\end{equation}
\end{subequations}
where $\partial_\xi f$ denotes the partial gradient with respect to $\xi$.

\begin{defin}[Bulk energy and admissible displacements]\label{302}
Let $K$ be a compact subset of $\cl\Omega$.
For every $u\in L^{1,p}(\Omega\setminus K)$, the bulk energy we will consider is defined by
\begin{equation*}
\cE_f(u,K):=\int_{\Omega\setminus K} f(x,\nabla u)\,\de x
\end{equation*}
Given $w\in W^{1,p}(\Omega)$, the set of admissible displacements determined by $w$ and $K$ is defined by
\begin{equation*}
\cA(w,K):=\{u\in L^{1,p}(\Omega\setminus K), u=w \text{ on } \partial_D\Omega\setminus K\},
\end{equation*}
where the equality on $\partial_D\Omega\setminus K$ is in the sense of traces.
\end{defin}
\begin{remark}\label{2321}
By properties \eqref{233}, the direct method of the calculus of variations implies that the functional $\cE_{f}(\cdot,K)$ has a minimizer in $\cA(w,K)$ and, by \eqref{233b}, any two minimizers have the same gradient.
\end{remark}

We fix once and for all $\alpha\in(1,2)$.
In our model, we assume that the energy dissipated by a crack $K$ (compact subset of $\cl\Omega$) is given by $\cH^\alpha(K)$.
It is convenient to introduce the total energy
\begin{equation}\label{307a}
\cE_{f}^{\tot}(u,K):=\cE_{f}(u,K)+\cH^\alpha(K).
\end{equation}

The following definition collects the properties of the collections of admissible cracks we are going to consider.
\begin{defin}[Admissible cracks]\label{238}
Let $\cK$ be a collection of compact subsets of $\cl\Omega$.
We say that $\cK$ is \emph{$\alpha$-admissible} if the following properties hold:
\begin{enumerate}
\item[(1)] (compactness) $\cK$ is compact for the Hausdorff metric;
\item[(2)] (lower semicontinuity) $\cH^\alpha$ is lower semicontinuous on $\cK$ with respect to the Hausdorff metric, i.e., $\displaystyle\cH^\alpha(K)\leq\smash{\liminf_{n\to\infty}}\,\cH^\alpha(K_n)$ whenever $K_n\to K$ in the Hausdorff metric;
\item[(3)] (connectedness property) there exists $m\geq1$ such that every element of $\cK$ has at most $m$ connected components;
\item[(4)] (extension property) given $H_n,H,K\in\cK$ such that $H_n\to H$ in the Hausdorff metric and $H\subset K$, then there exists a sequence $K_n$ in $\cK$ such that $K_n\to K$ in the Hausdorff metric, $K_n\supset H_n$ for all $n$, and
	\begin{equation}\label{220}
	\limsup_{n\to\infty} \cH^\alpha(K_n\setminus H_n)\leq\cH^\alpha(K\setminus H).
	\end{equation}
\end{enumerate}
\end{defin}

Let us fix $T>0$.
We are now ready to give the definition of quasistatic evolution corresponding to a time-dependent boundary datum $w\in AC([0,T];W^{1,p}(\Omega))$ on $\partial_D\Omega$.
\begin{defin}[Quasistatic evolution]\label{221a}
Let $f$ satisfy \eqref{233}, let $w\in AC([0,1]; W^{1,p}(\Omega))$, and let $\cK$ be an $\alpha$-admissible collection of compact subsets of $\cl\Omega$.
An \emph{irreversible quasistatic evolution of minimum energy configurations for $\cE_{f}^{\tot}$} corresponding to these data is a function $t\mapsto(u(t),K(t))$ satisfying the following conditions:
\begin{itemize}
\item[(a)] (irreversibility) $K(s)\subset K(t)$ for $0\leq s\leq t\leq T$;
\item[(b)] (global stability) for $0\leq t\leq T$ we have $K(t)\in\cK$, $\cH^\alpha(K(t))<+\infty$, $u(t)\in\cA(w(t),K(t))$, and $\cE_f^{\tot}(u(t),K(t))\leq\cE_f^{\tot}(\hat u,\hat K)$ for all $\hat K\in\cK$ with $\hat K\supset K(t)$ and for all $\hat u\in\cA(w(t),\hat K)$;
\item[(c)] (energy balance) $t\mapsto\cE_f^{\tot}(u(t),K(t))$ is absolutely continuous on $[0,T]$ and
\begin{equation*}
\frac{\de}{\de t}\cE_f^{\tot}(w(t),K(t))=\langle\partial_\xi f(x,\nabla u(t)1_{\Omega\setminus K(t)}),\nabla \dot w(t)\rangle\quad \text{ for a.e. } t\in[0,T].
\end{equation*}
\end{itemize}
\end{defin}
Here and in the rest of the paper $\langle\cdot,\cdot\rangle$ denotes the duality product between $L^q$ and $L^p$.
For every set $A\subset\R2$ the characteristic function $1_A$ is defined by $1_A(x)=1$ if $x\in A$ and $1_A(x)=0$ if $x\notin A$.
If $v$ is a function defined on $A$, the product $v1_A$ denotes the extension of $v$ to zero outside of $A$.

The following theorem establishes an existence result for a quasistatic evolution for the energy $\cE_f^\tot$.
\begin{theorem}\label{228a}
Let $f$, $w$, and $\cK$ be as in Definition \ref{221a}, let $K_0\in\cK$ with $\cH^\alpha(K_0)<+\infty$, and let $u_0\in\cA(w(0),K_0)$.
Assume that 
\begin{equation*}
\cE_f^{\tot}(u_0,K_0)\leq\cE_f^{\tot}(\hat u,\hat K)\quad \text{for all $\hat K\in\cK$ with $\hat K\supset K_0$ and for all $\hat u\in\cA(w(0),\hat K)$.}
\end{equation*}
Then there exists an irreversible quasistatic evolution $t\mapsto(u(t),K(t))$ for $\cE_f^{\tot}$ such that $K(0)=K_0$ and $u(0)=u_0$.
\end{theorem}

\section{Examples of $\boldsymbol{\alpha}$-admissible collections of compact sets}\label{sect3}
In this section we construct a class of $\alpha$-admissible compact subsets of $\cl\Omega$ which are not contained in a prescribed crack path.
As a matter of fact, the class $\cK$ that we are going to construct will be invariant under rotations.

We fix a simple curve with Hausdorff dimension $\alpha\in(1,2)$ represented as the image of an injective $1/\alpha$-H\"older continuous function $\gamma$. 
To be precise, we assume that $\gamma\colon[0,\ell]\to\R2$ is continuous, that $\gamma(0)=0$, that
\begin{equation}\label{100}
c_\gamma|s_1-s_2|^{1/\alpha}\leq|\gamma(s_1)-\gamma(s_2)|\leq C_\gamma|s_1-s_2|^{1/\alpha} \qquad\text{for all $s_1,s_2\in[0,\ell]$}
\end{equation}
for suitable constants $C_\gamma>c_\gamma>0$, and that
\begin{equation}\label{101}
\cH^\alpha(\gamma[s_1,s_2])=s_2-s_1\qquad\text{for $0\leq s_1\leq s_2\leq\ell$}.
\end{equation}
Here and henceforth we use the notation $\gamma[s_1,s_2]:=\{\gamma(s): s_1\leq s\leq s_2\}$.

A well-known example is given by the von Koch curve, for which the constants $c_\gamma$ and $C_\gamma$ are explicitly computed in \cite{Ponomarev2007}.

\begin{defin}\label{235}
Let $\gamma\colon[0,\ell]\to\R2$ be a continuous function satisfying \eqref{100} and \eqref{101}. 
We define $\cK(\gamma)$ as the collection of all sets $K$ of the form
\begin{equation}\label{111}
K=(\psi+R\gamma)[0,a],
\end{equation}
where $a\in[0,\ell]$, $R$ is an orthogonal matrix, and $\psi\colon[0,\ell]\to\R2$ is a Lipschitz function with Lipschitz constant $L:=\tfrac12c_\gamma\ell^{-1+1/\alpha}$ such that $\psi(0)=0$.
\end{defin}

\begin{remark}\label{242}
By the triangle inequality and the specific value of the Lipschitz constant $L$, the functions $\psi+R\gamma$ used in the definition of the class $\cK(\gamma)$ satisfy
\begin{equation}\label{255}
|(\psi+R\gamma)(s_1)-(\psi+R\gamma)(s_2)|\geq\frac{c_\gamma}2|s_1-s_2|^{1/\alpha}\quad \text{for all $s_1,s_2\in[0,\ell]$,}
\end{equation}
which, in particular, implies injectivity.
\end{remark}

\begin{proposition}\label{241}
Let $\gamma\colon[0,\ell]\to\R2$ be a function satisfying \eqref{100} and \eqref{101}, with $\gamma(0)=0$.
Then the collection $\cK(\gamma)$ is $\alpha$-admissible.
\end{proposition}
To prove Proposition \ref{241} the following lemma is required.
It shows that adding a Lipschitz function to $\gamma$ does not change the Hausdorff measure of the image, thus preserving condition \eqref{101}.
\begin{lemma}\label{101a}
Let $\gamma$, $\psi$, and $R$ be as in Definition \ref{235}.
Then
\begin{equation}\label{102}
\cH^\alpha((\psi+R\gamma)[s_1,s_2])= s_2-s_1
\end{equation}
for every $0\leq s_1\leq s_2\leq \ell$.
\end{lemma}
\begin{proof}
Since $\cH^\alpha(\gamma[s_1,s_2])=\cH^\alpha(R\gamma[s_1,s_2])$ and $R\gamma$ still satisfies \eqref{100} and \eqref{101}, it is not restrictive to assume $R=I$.

Let us fix $0\leq s_1\leq s_2\leq \ell$. We will start by proving that
\begin{equation}\label{102a}
\cH^\alpha((\psi+\gamma)[s_1,s_2])\leq \cH^\alpha(\gamma[s_1,s_2])
\end{equation}
Let $\eps>0$, $\delta>0$.
By definition of the Hausdorff measure, there exists a sequence $A_i$ of subsets of $\R2$ such that $\gamma[s_1,s_2]\subseteq\cup_{i=1}^\infty A_i$, $\delta_i:=\diam A_i<\delta$, and $\sum_{i=1}^\infty \omega_\alpha\delta_i^\alpha < \cH^\alpha(\gamma[s_1,s_2])+\eps$, where $\omega_\alpha$ is the usual constant for the $\alpha$-dimensional Hausdorff measure.

Define the set $S_i:=\{s\in[0,\ell]:\gamma(s)\in A_i\}$.
By \eqref{100} it is easy to see that $\diam S_i\leq\delta_i^\alpha/c_\gamma^\alpha$, hence $\sigma_i:=\diam \psi(S_i)\leq L\delta_i^\alpha/c_\gamma^\alpha$, which implies that $\sum_{i=1}^\infty \omega_\alpha\sigma_i< L(\cH^\alpha(\gamma[s_1,s_2])\allowbreak+\eps)/c_\gamma^\alpha$, where $L$ is as in Definition \ref{235}.

Consider now $\hat A_i:=A_i+\psi(S_i)$.
It is immediate to show that $\hat A_i$ is a covering for $(\psi+\gamma)[s_1,s_2]$, and that $\diam\hat A_i<\hat\delta$, where $\hat\delta\to0$ as $\delta\to0$.
Indeed, $\diam\hat A_i\leq\delta_i+\sigma_i\leq \delta_i+L\delta_i^\alpha/c_\gamma^\alpha< \delta+L\delta^\alpha/c_\gamma^\alpha=:\hat\delta$.

Using the convexity of the function $t\mapsto t^\alpha$ and the estimates for $\sigma_i$ and $\sum_{i=1}^\infty \sigma_i$, for $0<\eta<1$ we get
\begin{equation*}
\begin{split}
\sum_{i\in I} \omega_\alpha(\diam\hat A_i)^\alpha \leq & \sum_{i\in I} \omega_\alpha(\delta_i+\sigma_i)^\alpha\leq \frac{\omega_\alpha}{(1-\eta)^{\alpha-1}}\sum_{i\in I} \delta_i^\alpha +\frac{\omega_\alpha}{\eta^{\alpha-1}} \sum_{i\in I} \sigma_i^{\alpha-1}\sigma_i \\
\leq & (\cH^\alpha(\gamma[s_1,s_2])+\eps)\bigg(\frac1{(1-\eta)^{\alpha-1}}+\frac{\delta^{\alpha(\alpha-1)}L^\alpha}{\eta^{\alpha-1}c_\gamma^{\alpha^2}}\bigg).
\end{split}
\end{equation*}
By sending first $\delta\to0$ and then $\eps,\eta\to0$ and using the definition of Hausdorff measure, we obtain \eqref{102a}.

To prove the opposite inequality, we note that the function $\psi+\gamma$ satisfies \eqref{255}, which has the same structure as \eqref{100} with $c_\gamma$ replaced by $c_\gamma/2$.
Hence, we can apply the previous argument with $\gamma$ replaced by $\psi+\gamma$ and $\psi$ replaced by $-\psi$ and we obtain $\cH^\alpha(\gamma[s_1,s_2])\leq\cH^\alpha((\psi+\gamma)[s_1,s_2])$.
Together with \eqref{102a} this gives $\cH^\alpha(\gamma[s_1,s_2])=\cH^\alpha((\psi+\gamma)[s_1,s_2])$.
The conclusion follows from \eqref{101}.
\end{proof}

\begin{proof}[Proof of Proposition \ref{241}]
We have to show that $\cK(\gamma)$ satisfies properties (1)-(4) in Definition \ref{238}.
To prove the compactness property (1), let us fix a sequence $K_n$ in $\cK(\gamma)$.
By \eqref{111} there exist a sequence $\psi_n$ of Lipschitz functions with constant $L$, a sequence $R_n$ of orthogonal matrices, and a sequence $a_n$ in $[0,\ell]$ such that $K_n=(\psi_n+R_n\gamma)[0,a_n]$.
By Arzel\`a-Ascoli Theorem, passing to a subsequence, we may assume that $\psi_n\to\psi$ uniformly, $R_n\to R$, and $a_n\to a$, where $\psi$ is a Lipschitz function with constant $L$, $R$ is an orthogonal matrix, and $a\in[0,\ell]$.
Thus $K_n\to K:=(\psi+R\gamma)[0,a]\in\cK(\gamma)$ in the Hausdorff metric.

Let us prove the lower semicontinuity property (2).
Let $K_n\to K$ in the Hausdorff metric with $K_n,K\in\cK(\gamma)$.
We have to prove that
\begin{equation}\label{252}
\cH^\alpha(K)\leq\lim_{n\to\infty} \cH^\alpha(K_n),
\end{equation}
assuming that the limit exists.
By \eqref{111} we have $K_n=(\psi_n+R_n\gamma)[0,a_n]$ with $\psi_n$, $R_n$, $a_n$ as in the previous step.
As in the proof of property (1) we have that $K=(\psi+R\gamma)[0,a]$, where $\psi$ is a suitable Lipschitz function with constant $L$, $R$ is a suitable orthogonal matrix, and $a$ is the limit of a suitable subsequence of $a_n$.
By Lemma \ref{101a} we have $\cH^\alpha(K_n)=a_n$ and $\cH^\alpha(K)=a$.
This concludes the proof of \eqref{252}.

Property (3) about connectedness, with $m=1$, follows immediately from the continuity of $\gamma$ and $\psi$.

To prove the extension property (4), we fix $H_n,H,K\in\cK(\gamma)$ such that $H_n\to H$ in the Hausdorff metric and $H\subset K$.
By \eqref{111} there exist $\phi_n,\phi,\psi$ Lipschitz functions with constant $L$, $R_n,R,S$ orthogonal matrices, and $a_n,a,b\in[0,\ell]$ such that $\phi_n(0)=\phi(0)=\psi(0)=0$, $H_n=(\phi_n+R_n\gamma)[0,a_n]$, $H=(\phi+R\gamma)[0,a]$, and $K=(\psi+S\gamma)[0,b]$.
From Lemma \ref{101a}, it follows that $\cH^\alpha(H)=a$ and $\cH^\alpha(K)=b$.
Since $H\subset K$, we conclude that $a\leq b$.

Let us prove that $S=R$. 
Notice that $(\phi+R\gamma)[0,s]\subset (\phi+R\gamma)[0,a]=H \subset K$ for every $s\in[0,a]$.
Now, the set $(\phi+R\gamma)[0,s]$ is connected (since it is the continuous image of a connected set), contains $0$, and satisfies $\cH^\alpha((\phi+R\gamma)[0,s])=s$ by Lemma \ref{101a}.
By Lemma \ref{117} below, it follows that $(\phi+R\gamma)[0,s]=(\psi+S\gamma)[0,s]$ for all $s\in[0,a]$. 
Let us prove now that this implies 
\begin{equation}\label{253}
(\phi+R\gamma)(s)=(\psi+S\gamma)(s)\qquad \text{for all $s\in[0,a]$.}
\end{equation}
Suppose, by contradiction, that there esists $s_0\in[0,a]$ such that $\phi(s_0)+R\gamma(s_0)\neq\psi(s_0)+S\gamma(s_0)$.
Since $\phi(s_0)+R\gamma(s_0)\in (\phi+R\gamma)[0,s_0]= (\psi+S\gamma)[0,s_0]$, there exists $s_1\in[0,s_0)$ such that $\phi(s_0)+R\gamma(s_0)=\psi(s_1)+S\gamma(s_1)$.
On the other hand, since $\psi(s_1)+S\gamma(s_1)\in (\psi+S\gamma)[0,s_1]=(\phi+R\gamma)[0,s_1]$, there exists $s_2\in[0,s_1]$ such that $\phi(s_2)+R\gamma(s_2)=\psi(s_1)+S\gamma(s_1)=\phi(s_0)+R\gamma(s_0)$.
Since $s_2\leq s_1<s_0$, this violates the injectivity condition \eqref{255}. 
This contradiction proves \eqref{253}.

By applying $S^{-1}$ to both sides of \eqref{253}, we deduce that for every $s\in[0,a]$ we have $S^{-1}\phi(s)+S^{-1}R\gamma(s) = S^{-1}\psi(s)+\gamma(s)$, which implies 
\begin{equation*}
S^{-1}(\phi(s)-\psi(s))=(I-S^{-1}R)\gamma(s).
\end{equation*}
Since $R$ and $S$ are orthogonal matrices, if $I-S^{-1}R\neq0$, then it is an invertible matrix and we deduce $\gamma(s)=(I-S^{-1}R)^{-1}S^{-1}(\phi(s)-\psi(s))$ for every $s\in[0,a]$. 
This is impossible, since the right-hand side is Lipschitz continuous, while $\gamma$ cannot be Lipschitz continuous in $[0,a]$ because of \eqref{100}.
Therefore, it must be $S=R$.
It follows immediately from \eqref{253} that $\phi=\psi$ on $[0,a]$.
The same argument shows also that $R$ and $a$ are uniquely determined by $H$, and so is $\phi$ on $[0,a]$.

By Arzel\`a-Ascoli Theorem, up to a subsequence, we have that $\phi_n\to\hat\phi$ uniformly, $R_n\to\hat R$, and $a_n\to\hat a$, where $\hat\phi$ is a Lipschitz function with $\hat\phi(0)=0$ and Lipschitz constant $L$, $\hat R$ is an orthogonal matrix, and $\hat a\in[0,1]$. 
Since $H_n\to H$ in the Hausdorff metric, we have that $H=(\hat\phi+\hat R\gamma)[0,\hat a]$.
By uniqueness, $\hat R=R$, $\hat a=a$, and $\hat\phi=\phi$ on $[0,a]$.
As the limit does not depend on the subsequence, we conclude that the full sequence $(R_n,a_n)$ converges to $(R,a)$ and $\phi_n|_{[0,a_n]}\to\phi|_{[0,a]}$ uniformly.

We now define
\begin{equation*}
\psi_n(s):=\begin{cases}
\phi_n(s) & 0\leq s\leq a_n, \\
\psi(s)+\phi_n(a_n)-\psi(a_n) & a_n<s\leq\ell,
\end{cases}
\end{equation*}
and $K_n:=(\psi_n+R_n\gamma)[0,b_n]$, where $b_n:=\max\{a_n, b\}$.
Since $\psi_n$ are Lipschitz with constant $L$, the sets $K_n$ belong to $\cK(\gamma)$.
Since $\psi=\phi$ on $[0,a]$ and $\phi_n|_{[0,a_n]}\to\phi|_{[0,a]}$ uniformly, we have that $\psi_n\to\psi$ uniformly on $[0,\ell]$.
Since, in addition, $R_n\to R$ and $b_n\to b$ and $K=(\psi+R\gamma)[0,b]$, we conclude that $K_n\to K$ in the Hausdorff metric.
By construction, we have that $K_n\supset H_n$ and $K_n\setminus H_n=(\psi_n+R_n\gamma)(a_n,b_n]$, hence by \eqref{102} $\cH^\alpha(K_n\setminus H_n)=b_n-a_n$.
Since $K\setminus H=(\psi+R\gamma)(a,b]$, by \eqref{102} we have $\cH^\alpha(K\setminus H)=b-a$, thus inequality \eqref{220} is satisfied.
This concludes the proof.
\end{proof}

\begin{lemma}\label{117}
Let $\gamma$ and $\psi$ be as in Definition \ref{235} and let $S$ be an orthogonal matrix.
Let $0<\sigma<\ell$ and let $H_\sigma$ be a compact connected set such that $0\in H_\sigma$, $H_\sigma\subset(\psi+S\gamma)[0,\ell]$, and $\cH^\alpha(H_\sigma)=\sigma$. 
Then $H_\sigma=(\psi+S\gamma)[0,\sigma]$.
\end{lemma}
\begin{proof}
The result easily follows from known properties of the topology of the line.
Indeed, the map $(\psi+S\gamma):[0,\ell]\to (\psi+S\gamma)[0,\ell]$ is continuous and injective by Remark \ref{242}, and therefore it is a homeomorphism.
Thus, since $(\psi+S\gamma)^{-1}(H_\sigma)$ is compact, connected, contains the origin, and is contained in $[0,\ell]$, it must have the form $[0,\tau]$ for a certain $\tau$.
Therefore, $(\psi+S\gamma)[0,\tau]=H_\sigma$.
By \eqref{102}, we have $\tau=\cH^\alpha((\psi+S\gamma)[0,\tau])=\cH^\alpha(H_\sigma)=\sigma$, which concludes the proof.
\end{proof}

\section{Convergence of minimizers}\label{270}
In this section we prove the following result on the strong convergence of the solutions to minimum problems in $\Omega\setminus K_n$.
\begin{theorem}\label{DMT5.1}
Let $K_n$ be a sequence of compact subsets of $\cl\Omega$ satysfying condition (3) 
of Definition \ref{238}.
Assume that $K_n$ converges to $K\subset\cl\Omega$ in the Hausdorff metric and that $\cL^2(K)=0$, where $\cL^2$ is the two-dimensional Lebesbue measure. 
Let $f$ satisfy \eqref{233}, let $w_n\to w$ strongly in $W^{1,p}(\Omega)$, and let $u_n$ and $u$ be the solutions to the minimum problems 
\begin{equation*}
\min\{\cE_{f}(v,K_n):v\in\cA(w_n,K_n)\} \qquad\text{and}\qquad\min\{\cE_{f}(v,K):v\in\cA(w,K)\}.
\end{equation*}
Then $\nabla u_n1_{\Omega\setminus K_n}\to\nabla u1_{\Omega\setminus K}$ strongly in $L^p(\Omega;\R2)$.
\end{theorem}

To prove Theorem \ref{DMT5.1} we need the following lemma, which is related to the notion of Mosco convergence (see \cite{Mosco1969}).
\begin{lemma}\label{nd100}
Let $K_n$, $K$, $w_n$, and $w$ be as in Theorem \ref{DMT5.1} and let $v\in W^{1,p}(\Omega\setminus K)$ with $v=w$ on $\partial_D\Omega\setminus K$.
Then there exists a sequence $v_n\in W^{1,p}(\Omega\setminus K_n)$, with $v_n=w_n$ on $\partial_D\Omega\setminus K_n$, such that
\begin{subequations}\label{nd103}
\begin{eqnarray}
v_n1_{\Omega\setminus K_n} & \to & v1_{\Omega\setminus K}\qquad \text{strongly in $L^p(\Omega)$,} \\ 
\nabla v_n1_{\Omega\setminus K_n} & \to & \nabla v1_{\Omega\setminus K}\qquad \text{strongly in $L^p(\Omega;\R2)$.} \label{nd102}
\end{eqnarray}
\end{subequations}
\end{lemma}

\begin{proof}
Let us fix a large open ball $B$ containing $\cl\Omega$ and let us define $A_n:=B\setminus(K_n\cup\partial_N\Omega)$ and $A:=B\setminus(K\cup\partial_N\Omega)$.
By standard extension theorems, it is possible to extend $w_n$ and $w$ to the ball $B$ so that $w_n,w\in W^{1,p}(B)$ and $w_n\to w$ strongly in $W^{1,p}(B)$.
Let us define $z\in W^{1,p}(A)$ by $z:=v$ on $\Omega\setminus K$ and $z:=w$ on $B\setminus\Omega$.

Since $K_n\cup\partial_N\Omega\to K\cup\partial_N\Omega$ in the Hausdorff metric and consequently $\limsup_{n} \cL^2(K_n\cup\partial_N\Omega)\leq\cL^2(K\cup\partial_N\Omega)=0$, we can apply \cite[Theorem 4.2]{DalMasoEbobissePonsiglione2003} and we obtain that there exists a sequence $z_n\in W^{1,p}(A_n)$ such that 
\begin{subequations}\label{nd104}
\begin{eqnarray}
z_n1_{A_n} & \to & z1_{A}\qquad \text{strongly in $L^p(B)$,} \\ 
\nabla z_n1_{A_n}  &\to&  \nabla z1_{A}\qquad \text{strongly in $L^p(B;\R2)$.} 
\end{eqnarray}
\end{subequations}
In particular, $z_n\to w$ strongly in $W^{1,p}(B\setminus\cl\Omega)$.
Since $w_n\to w$ strongly in $W^{1,p}(B\setminus\cl\Omega)$, we have $z_n-w_n\to0$ strongly in $W^{1,p}(B\setminus\cl\Omega)$, therefore, by using a standard extension operator, we can find a sequence $\varphi_n\in W^{1,p}(B)$ such that $\varphi_n=z_n-w_n$ in $B\setminus\cl\Omega$ and $\varphi_n\to0$ strongly in $W^{1,p}(B)$.

We now define $v_n:=z_n-\varphi_n\in W^{1,p}(A_n)$.
By construction, the outer trace of $v_n$ on $\partial\Omega$ (\emph{i.e.}, the trace from $B\setminus\cl\Omega$) is equal to $w_n$, so the inner trace (\emph{i.e.}, the trace from $\Omega\setminus K_n$) satisfies $v_n=w_n$ on $\partial_D\Omega\setminus K_n$.
Moreover, the strong convergence of $\varphi_n$ to $0$ in $W^{1,p}(B)$ implies that \eqref{nd103} follows from \eqref{nd104}.
\end{proof}

\begin{proof}[Proof of Theorem \ref{DMT5.1}]
By using the boundary datum $w_n$ as a competitor in the minimum problem solved by $u_n$, we obtain $\cE_f(u_n,K_n)\leq\cE_f(w_n,K_n)$.
Since $w_n$ is bounded in $W^{1,p}(\Omega)$, we deduce from \eqref{233c} that $\nabla u_n1_{\Omega\setminus K_n}$ are bounded in $L^p(\Omega;\R2)$.

Therefore, there exists $\Psi\in L^p(\Omega;\R2)$ such that, up to a subsequence,
\begin{equation*}
\nabla u_n1_{\Omega\setminus K_n}\wto\Psi  \quad\text{in $L^p(\Omega;\R2)$.} 
\end{equation*}
Since $K_n\to K$ in the Hausdorff metric, it is easy to see that there exists $u^*\in L^{1,p}(\Omega\setminus K)$, with $u^*=w$ on $\partial_D\Omega\setminus K$, such that $\Psi=\nabla u^*1_{\Omega\setminus K}$ (see, e.g., \cite[Lemma 4.1]{DalMasoToader2002}).

We now want to prove that
\begin{equation}\label{nd110}
\cE_f(u^*,K)\leq\cE_f(v,K)\qquad \text{for all $v\in\cA(w,K)$}.
\end{equation}
We fix $v\in\cA(w,K)$ and, for every integer $k>0$, we consider the truncation $v^k:=(v\wedge(w+k))\vee(w-k)$.
It is easy to see that $v^k\in W^{1,p}(\Omega\setminus K)$, that $v^k=w$ on $\partial_D\Omega\setminus K$, and that $\nabla v^k\to\nabla v$ strongly in $L^p(\Omega\setminus K;\R2)$ as $k\to\infty$.
By \eqref{233} this implies that $\cE_f(v^k,K)\to\cE_f(v,K)$.
Therefore, to prove \eqref{nd110} it is enough to show that 
\begin{equation}\label{nd111}
\cE_f(u^*,K)\leq\cE_f(v^k,K)\qquad \text{for every $k$}.
\end{equation}
Fix an integer $k>0$ and apply Lemma \ref{nd100} with $v$ replaced by $v^k$ to obtain a sequence $v_n^k\in\cA(w_n,K_n)$, satisfying \eqref{nd102} with $v_n$ replaced by $v_n^k$ and $v$ replaced by $v^k$.
By the minimality of $u_n$ we have, for every $n$,
\begin{equation*}
\int_{\Omega} f(x,\nabla u_n1_{\Omega\setminus K_n})\,\de x\leq \int_{\Omega} f(x,\nabla v_n^k1_{\Omega\setminus K_n})\,\de x.
\end{equation*}
Since $\Phi\mapsto\int_\Omega f(x,\Phi)\,\de x$ is lower semicontinuous in $L^p(\Omega;\R2)$ with respect to weak convergence and continuous with respect to strong convergence, we obtain \eqref{nd111}, and therefore \eqref{nd110}.
By uniqueness, $\nabla u^*=\nabla u$ a.e.\@ in $\Omega\setminus K$.

We now prove that
\begin{equation}\label{nd113}
\int_\Omega f(x,\nabla u_n1_{\Omega\setminus K_n})\,\de x\to \int_\Omega f(x,\nabla u1_{\Omega\setminus K})\,\de x.
\end{equation}
Let us fix an integer $k>0$, and let $u^k:=(u\wedge(w+k))\vee(w-k)\in W^{1,p}(\Omega\setminus K)$. 
As before, we observe that $u^k=w$ on $\partial_D\Omega\setminus K$, and that $\nabla u^k\to\nabla u$ strongly in $L^p(\Omega\setminus K;\R2)$ as $k\to\infty$.
We apply now Lemma \ref{nd100} with $v$ replaced by $u^k$ to obtain a sequence $z_n^k\in\cA(w_n,K_n)$, satisfying \eqref{nd102} with $v_n$ replaced by $z_n^k$ and $v$ replaced by $u^k$.
By the minimality of $u_n$ we have, for every $n$,
\begin{equation*}
\int_{\Omega} f(x,\nabla u_n1_{\Omega\setminus K_n})\,\de x\leq \int_{\Omega} f(x,\nabla z_n^k1_{\Omega\setminus K_n})\,\de x.
\end{equation*}
Since $\Phi\mapsto\int_\Omega f(x,\Phi)\,\de x$ is continuous with respect to strong convergence, passing to the limit, first as $n\to\infty$ and then as $k\to\infty$, we obtain
\begin{equation*}
\limsup_{n\to\infty} \int_{\Omega} f(x,\nabla u_n1_{\Omega\setminus K_n})\,\de x\leq \int_{\Omega} f(x,\nabla u1_{\Omega\setminus K})\,\de x.
\end{equation*}
The opposite inequality with $\liminf$ follows from the weak convergence of $\nabla u_n1_{\Omega\setminus K_n}$ to $\nabla u1_{\Omega\setminus K}$ by lower semicontinuity.
This concludes the proof of \eqref{nd113}.

We can now apply \cite[Theorem 3]{Visintin}, obtaining that
\begin{equation}\label{nd116}
\nabla u_n1_{\Omega\setminus K_n} \to \nabla u1_{\Omega\setminus K} \qquad\text{in measure on $\Omega$}
\end{equation}
and that
\begin{equation*}
f(x,\nabla u_n1_{\Omega\setminus K_n})\to f(x,\nabla u1_{\Omega\setminus K}) \qquad\text{strongly in $L^1(\Omega)$.}
\end{equation*}
Therefore, by the lower estimate in \eqref{233c} we deduce that $|\nabla u_n1_{\Omega\setminus K_n}|^p$ is equiintegrable, which, together with \eqref{nd116}, implies that $\nabla u_n1_{\Omega\setminus K_n} \to \nabla u1_{\Omega\setminus K}$ strongly in $L^p(\Omega;\R2)$.
\end{proof}

\section{Quasistatic evolution}\label{555}
In this section we prove Theorem \ref{228a} by a time discretization procedure.
For each $n\in\N{}$, consider a partition $\{t_n^i\}_{i=0}^n$ of the time interval $[0,T]$ such that $0=t_n^0<t_n^1<\cdots<t_n^n=T$ and
\begin{equation}\label{556}
\max_{1\leq i\leq n}(t_n^i-t_n^{i-1})\to0,\quad\text{as $n\to\infty$}.
\end{equation}
Set $w_n^i:=w(t_n^i)$, $u_n^0:=u_0$, $K_n^{0}:=K_0$ for all $n$ and $i$.
We define $u_n^i$ and $K_n^i$ by induction on $i$: for every $i=1,\ldots,n$, the pair $(u_n^i,K_n^i)$ is a solution to the minimum problem
\begin{equation}\label{510}
\min_{(u,K)}\{\cE_f^\tot(u,K):u\in\cA(w_n^i,K),\,K\in\cK,\, K\supset K_n^{i-1}\}.
\end{equation}
\begin{proposition}\label{511}
Under conditions (1) and (2) of Definition \ref{238}, the minimum problem \eqref{510} has a solution $(u_n^i,K_n^i)$ and $\cH^\alpha(K_n^i)<+\infty$.
\end{proposition}
\begin{proof}
Assume, by induction, that $\cH^\alpha(K_n^{i-1})<+\infty$. Then the infimum in \eqref{510} is less than or equal to $\cE_f(w_n^i,K_n^{i-1})+\cH^\alpha(K_n^{i-1})<+\infty$.
Consider a minimizing sequence $(v_j,K_j)$ for \eqref{510}.
It must be $v_j\in \cA(w_n^{i},K_j)$, $K_j\in\cK$, and $K_j\supset K_n^{i-1}$ for all $j$.
Since $\cE_f^\tot(v_j,K_j)$ is uniformly bounded, so is $\cH^\alpha(K_j)$; by property (1) of Definition \ref{238}, there exists $K_n^{i}\in\cK$, with $K_n^{i}\supset K_n^{i-1}$, such that (up to a subsequence) $K_j\to K_n^{i}$ in the Hausdorff metric.

Arguing as in the proof of Theorem \ref{DMT5.1}, we can show that there exists $u_n^i\in \cA(w_n^i,K_n^i)$ such that
\begin{equation*}
\nabla v_j1_{\Omega\setminus K_j}\wto\nabla u_n^{i}1_{\Omega\setminus K_n^{i}}\quad \text{weakly in $L^p(\Omega;\R2)$,}
\end{equation*}
and therefore, by lower semincontinuity,
\begin{equation*}
\int_{\Omega\setminus K_n^{i}} f(x,\nabla u_n^{i})\,\de x\leq \liminf_{j\to\infty}\int_{\Omega\setminus K_j} f(x,\nabla v_j)\,\de x.
\end{equation*}
Thus, by property (2) of Definition \ref{238}, $(u_n^{i},K_n^{i})$ is a minimizer of problem \eqref{510} and $\cH^\alpha(K_n^i)<+\infty$.
\end{proof}

Since $w_n^i\in\cA(w_n^i,K_n^{i})$ and $K_n^i\supset K_n^{i-1}$, by \eqref{510} we have $\cE_f(u_n^i,K_n^i)+\cH^\alpha(K_n^i)\leq \break \cE_f(w_n^i,K_n^i)+\cH^\alpha(K_n^i)$.
By Proposition \ref{511}, we have $\cH^\alpha(K_n^i)<+\infty$, hence $\cE_f(u_n^i,K_n^i)\leq\cE_f(w_n^i,K_n^i)$.
Since $w\in AC([0,T];W^{1,p}(\Omega))$, estimates \eqref{233c} on $f$ give that
\begin{equation}\label{562}
\nabla u_n^i1_{\Omega\setminus K_n^i}\;\;\text{is bounded in $L^p(\Omega;\R2)$ uniformly with respect to $n$ and $i$.}
\end{equation}

Let us construct the (right-continuous) piecewise constant interpolants $u_n$ and $K_n$ by defining 
\begin{equation}\label{561}
u_n(t):=u_n^i,\qquad K_n(t):=K_n^i,\qquad\text{for $t\in[t_n^i,t_n^{i+1})$.}
\end{equation}
It follows from \eqref{562} that 
\begin{equation}\label{562a}
\nabla u_n(t)1_{\Omega\setminus K_n(t)}\;\;\text{is bounded in $L^p(\Omega;\R2)$ uniformly with respect to $n$ and $t$.}
\end{equation}
To continue the proof, we need the following extension of Helly's Theorem, proved in \cite[Theorem 6.3]{DalMasoToader2002}.

\begin{proposition}\label{560}
Let $K_n$ be a sequence of increasing functions from $[0,T]$ into the collection $\cK(\cl\Omega)$ of compact subsets of $\cl\Omega$.
Then there exist a subsequence, still denoted by $K_n$, and an increasing function $K\colon[0,T]\to\cK(\cl\Omega)$, such that $K_n(t)\to K(t)$ in the Hausdorff metric for every $t\in[0,T]$.
\end{proposition}

We are now ready to prove one of the main results of the paper.
\begin{proof}[Proof of Theorem \ref{228a}]
Let $u_n(t)$ and $K_n(t)$ be defined as in \eqref{561}, with $u_n^i$ and $K_n^i$ defined through \eqref{510}.
By Proposition \ref{560} there exist compact subsets $K(t)$ of $\cl\Omega$ and a subsequence, independent of $t$ and still denoted by $K_n$, such that $K_n(t)\to K(t)$ in the Hausdorff metric for every $t\in[0,T]$.
Notice that $t\mapsto K(t)$ is increasing and $K(t)\in\cK$ for every $t\in[0,T]$ by property (1) of Definition \ref{238}.
This gives the irreversibility condition (a) of Definition \ref{221a}.

Arguing as in \cite[Lemma 6.1]{DalMasoFrancfortToader2005} with $\cW=f$ and $\cF=\cG=0$, we obtain that
\begin{equation}\label{575}
\cE_f^\tot(u_n(t),K_n(t))\leq\cE_f^\tot(u_0,K_0)+\int_0^t \scp{\partial_\xi f(x,\nabla u_n(s)1_{\Omega\setminus K_n(s)})}{\nabla\dot w(s)}\de s+r_n,
\end{equation}
where the remainder $r_n\to0$ as $n\to\infty$.
This implies, in particular, that
\begin{equation*}
\cH^\alpha(K_n(t))\;\;\text{is bounded uniformly with respect to $n$ and $t$.}
\end{equation*}

To prove condition (b) of Definition \ref{221a}, let us fix $t\in(0,T]$.
The minimality of $(u_n^i,K_n^i)$ implies that 
\begin{equation}\label{563}
\cE_f(u_n(t),K_n(t)) \leq \cE_f(\hat u,\hat K)+\cH^\alpha(\hat K\setminus K_n(t)),
\end{equation}
for every $\hat K\in\cK$ with $\hat K\supset K_n(t)$ and for every $\hat u\in\cA(w_n(t),\hat K)$.
Taking $\hat K=K_n(t)$ we deduce that $u_n(t)$ is a solution to the miminum problem
\begin{equation*}
\min\{\cE_f(v,K_n(t)): v\in\cA(w_n(t),K_n(t))\}.
\end{equation*}
Let $u(t)$ be a solution of 
\begin{equation*}
\min\{\cE_f(v,K(t)): v\in\cA(w(t),K(t))\}.
\end{equation*}
By Theorem \ref{DMT5.1}, we have
\begin{equation}\label{567}
\nabla u_n(t)1_{\Omega\setminus K_n(t)}\to\nabla u(t)1_{\Omega\setminus K(t)} \qquad\text{strongly in $L^p(\Omega;\R2)$,}
\end{equation}
which implies, in particular, that
\begin{equation}\label{568}
\cE_f(u_n(t),K_n(t))\to\cE_f(u(t),K(t)).
\end{equation}
Moreover, from \eqref{562a} and \eqref{567} it follows that
\begin{equation}\label{562b}
\nabla u(t)1_{\Omega\setminus K(t)}\;\;\text{is bounded in $L^p(\Omega;\R2)$ uniformly with respect to $t$.}
\end{equation}

Let now fix a set $\hat K$ as in part (b) of Definition \ref{221a}.
Then there exists a minimizer $\hat v$ of $\cE_f(\cdot,\hat K)$ in $\cA(w(t),\hat K)$.
Moreover, for every $\hat K_n\in\cK$, with $\hat K_n\supset K_n(t)$, by \eqref{563}
\begin{equation}\label{570}
\cE_f(u_n(t),K_n(t))\leq \cE_f(\hat v_{n},\hat K_n)+\cH^\alpha(\hat K_n\setminus K_n(t)),
\end{equation}
where $\hat v_{n}$ is the minimizer of $\cE_f(\cdot,\hat K_n)$ in $\cA(w_n(t),\hat K_n)$.
We can choose the sets $\hat K_n$ that satisfy the extension property (4) in Definition \ref{238}, so that 
\begin{equation}\label{572}
\limsup_{n\to\infty} \cH^\alpha(\hat K_n\setminus K_n(t))\leq\cH^\alpha(\hat K\setminus K(t)).
\end{equation}
By Theorem \ref{DMT5.1}, we have also
\begin{equation}\label{571}
\cE_f(\hat v_{n},\hat K_n)\to \cE_f(\hat v,\hat K).
\end{equation}
Therefore, combining \eqref{568}, \eqref{570}, \eqref{572}, and \eqref{571}, we obtain
\begin{equation*}
\cE_f(u(t),K(t))\leq \cE_f(\hat v,\hat K)+\cH^\alpha(\hat K\setminus K(t)),
\end{equation*}
which gives the global stability property (b) of Definition \ref{221a}.

It remains to prove property (c) of Definition \ref{221a}, which is clearly equivalent to
\begin{equation}\label{575c}
\cE_f^\tot(u(t),K(t))=\cE_f^\tot(u_0,K_0)+\int_0^t \scp{\partial_\xi f(x,\nabla u(s)1_{\Omega\setminus K(s)})}{\nabla\dot w(s)}\de s
\end{equation}
for every $t\in[0,T]$.
One inequality can be proved by passing to the limit in \eqref{575} thanks to \eqref{233} and \eqref{567}.
To prove the opposite inequality, given $t\in[0,T]$, for each $n\in\N{}$ consider a partition $\{s_n^i\}_{i=0}^n$ such that $0=s_n^0<s_n^1<\cdots<s_n^n=t$ and satisfying \eqref{556}.
Arguing as in \cite[Lemma 7.1]{DalMasoFrancfortToader2005} with $\cW=f$ and $\cF=\cG=0$, we obtain that
\begin{equation}\label{575a}
\cE_f^\tot(u(t),K(t))\geq\cE_f^\tot(u_0,K_0)+\sum_{i=1}^n\int_{s_n^{i-1}}^{s_n^i} \scp{\partial_\xi f(x,\nabla u(s_n^i))}{\nabla\dot w(s)}\de s-r_n(t),
\end{equation}
where the remainder $r_n(t)\to0$ as $n\to\infty$.
In order to pass to the limit in \eqref{575a}, we introduce the (left-continuous) piecewise constant interpolant $\hat u_n$ and $\hat K_n$ defined by $\hat u_n(t):=u(s_n^i)$ and $\hat K_n(t)=K(s_n^i)$ for $s_n^{i-1}<t\leq s_n^i$.
With this notation \eqref{575a} reads
\begin{equation}\label{575b}
\cE_f^\tot(u(t),K(t))\geq\cE_f^\tot(u_0,K_0)+\int_0^t \big\langle\partial_\xi f(x,\nabla\hat u_n(s)1_{\Omega\setminus\hat K_n(s)}),\nabla\dot w(s)\big\rangle\,\de s-r_n(t).
\end{equation}
Moreover, by \eqref{562b}
\begin{equation*}
\nabla\hat u_n(s)1_{\Omega\setminus\hat K_n(s)}\;\;\text{is bounded in $L^p(\Omega;\R2)$ uniformly with respect to $n$ and $s\in[0,t]$.}
\end{equation*}

Let us prove that
\begin{equation}\label{576}
\nabla\hat u_n(s)1_{\Omega\setminus \hat K_n(s)}\to\nabla u(s)1_{\Omega\setminus K(s)}\qquad\text{in $L^p(\Omega;\R2)$ for almost every $s\in[0,t]$.}
\end{equation}
Arguing as in the proof of \cite[Theorem 6.3]{DalMasoToader2002} we obtain that for almost every $s\in[0,t]$ we have $K(s_n)\to K(s)$ in the Hausdorff metric whenever $s_n\to s$.
By the minimality of $u(s)$ and $u(s_n^i)$, \eqref{576} follows from Theorem \ref{DMT5.1}.
In view of \eqref{233b} and \eqref{233d}, properties \eqref{567} and \eqref{576} imply that
\begin{equation*}
\partial_\xi f(x,\nabla u_n(s)1_{\Omega\setminus K_n(s)})\to\partial_\xi f(x,\nabla u(s)1_{\Omega\setminus K(s)})\qquad\text{strongly in $L^q(\Omega;\R2)$,}
\end{equation*}
\begin{equation*}
\partial_\xi f(x,\nabla\hat u_n(s)1_{\Omega\setminus\hat K_n(s)})\to\partial_\xi f(x,\nabla u(s)1_{\Omega\setminus K(s)})\qquad\text{strongly in $L^q(\Omega;\R2)$,}
\end{equation*}
for almost every $s\in[0,t]$.
By the Dominated Convergence Theorem, inequalities \eqref{575} and \eqref{575b} imply \eqref{575c}.
\end{proof}

\section{The case of linearized elasticity}\label{sectlinel}
In this section we consider the case of planar elasticity. 
Theorem \ref{228a} can be generalized to the vectorial case by combining Theorem \ref{DMT5.1} and the results contained in \cite{Chambolle2003}.
This leads to Theorem \ref{228b} below.

We proceed by introducing the function space, the elastic energy, and the set of admissible displacements used in this analysis.
\begin{defin}[\cite{Chambolle2003}]\label{903}
Given an open subset $U$ of $\cl\Omega$, we define
\begin{equation*}
LD^2(U):=\{u\in L^2_\loc(U;\R2): Eu\in L^2(U;\M{2\times2}_\sym)\},
\end{equation*}
where $Eu:=\tfrac12(\nabla u+(\nabla u)^\top)$ is the symmetrized gradient of $u$ and $\M{2\times2}_\sym$ is the space of $2{\times}2$ symmetric matrices.
\end{defin}
We shall prove in Proposition \ref{a111} that $LD^2(U)=W^{1,2}(U;\R2)$ if $\partial U$ is Lipschitz. 
More in general, if $\partial U$ is Lipschitz near a point $x\in\partial U$ and $U$ lies locally on one side of $\partial U$, then there exists an open neighborhood $V$ of $x$ such that $U\cap V$ has a Lipschitz boundary and therefore $u\in W^{1,2}(U\cap V;\R2)$.
This allows us to define the trace of a function $u\in LD^2(U)$ on the locally Lipschitz part of the boundary $\partial U$.

For every $x\in\Omega$ the space-dependent elasticity tensor $\C(x)$ is a symmetric linear operator of $\M{2\times2}_\sym$ into itself, such that $x\mapsto\C(x)$ is measurable.
We assume the usual ellipticity estimates: there exist constants $0<\alpha_\C\leq \beta_\C<+\infty$ such that
\begin{equation}\label{904}
\alpha_\C |A|^2\leq\C(x)A{:}A\leq\beta_\C |A|^2,\qquad\text{for every $A\in\M{2\times2}_\sym$},
\end{equation}
where $|A|$ denotes the Euclidean norm of $A\in\M{2\times2}_\sym$ and $A{:}B$ is the Euclidean scalar product of $A,B\in\M{2\times2}_\sym$.

\begin{defin}[Elastic energy and admissible displacements]\label{302a}
Let $K$ be a compact subset of $\cl\Omega$.
For every $u\in LD^2(\Omega\setminus K)$, the elastic energy we will consider is defined by
\begin{equation*}
\cE_\C(u,K):=\frac12\int_{\Omega\setminus K} \C(x)Eu{:}Eu\,\de x
\end{equation*}
Given $w\in W^{1,2}(\Omega;\R2)$, the set of admissible displacements determined by $w$ and $K$ is defined by
\begin{equation*}
\cA_{LD^2}(w,K):=\{u\in LD^2(\Omega\setminus K): u=w \text{ on } \partial_D\Omega\setminus K\},
\end{equation*}
where the equality on $\partial_D\Omega\setminus K$ is in the sense of traces.
\end{defin}
In analogy to \eqref{307a}, we introduce the total energy
\begin{equation*}
\cE_\C^\tot(u,K):=\cE_\C(u,K)+\cH^\alpha(K).
\end{equation*}

\begin{defin}[Quasistatic evolution]\label{221b}
Let $\C$ satisfy \eqref{904}, let $w{\in} AC([0,1];W^{1,2}(\Omega;\R2))$, and let $\cK$ be an $\alpha$-admissible collection of compact subsets of $\cl\Omega$ acording to Definition \ref{238}.
An \emph{irreversible quasistatic evolution of minimum energy configurations for $\cE_{\C}^{\tot}$} corresponding to these data is a function $t\mapsto(u(t),K(t))$
satisfying the following conditions:
\begin{itemize}
\item[(a)] (irreversibility) $K(s)\subset K(t)$ for $0\leq s\leq t\leq T$;
\item[(b)] (global stability) for $0\leq t\leq T$ we have $K(t)\in\cK$, $u(t)\in\cA_{LD}(w(t),K(t))$, and $\cE_\C^{\tot}(u(t),K(t))\leq\cE_\C^{\tot}(\hat u,\hat K)$ for all $\hat K\in\cK$ with $\hat K\supset K(t)$ and for all $\hat u\in\cA_{LD^2}(w(t),\hat K)$;
\item[(c)] (energy balance) $t\mapsto\cE_\C^{\tot}(u(t),K(t))$ is absolutely continuous on $[0,T]$ and
\begin{equation*}
\frac{\de}{\de t}\cE_\C^{\tot}(w(t),K(t))=\langle \C(x)Eu(t)1_{\Omega\setminus K(t)},E\dot w(t)\rangle\quad \text{ for a.e. } t\in[0,T].
\end{equation*}
\end{itemize}
\end{defin}

\begin{theorem}\label{228b}
Let $\C$, $w$, and $\cK$ be as in Definition \ref{221b}, let $K_0\in\cK$ with $\cH^\alpha(K_0)<+\infty$, and let $u_0\in\cA_{LD^2}(w(0),K_0)$.
Assume that 
\begin{equation*}
\cE_\C^{\tot}(u_0,K_0)\leq\cE_\C^{\tot}(\hat u,\hat K)\quad \text{for all $\hat K\in\cK$ with $\hat K\supset K_0$ and for all $\hat u\in\cA_{LD^2}(w(0),\hat K)$.}
\end{equation*}
Then there exists an irreversible quasistatic evolution $t\mapsto(u(t),K(t))$ for $\cE_\C^{\tot}$ such that $K(0)=K_0$ and $u(0)=u_0$.
\end{theorem}
\begin{proof}
It is enough to repeat the proof of Theorem \ref{228a} in Section \ref{555} with obvious modifications, replacing Theorem \ref{DMT5.1} on the convergence of minimizers with Theorem \ref{DMT5.1a} below.
\end{proof}
\begin{theorem}\label{DMT5.1a}
Let $K_n$ be a sequence of compact subsets of $\cl\Omega$ satysfying conditions (3) 
of Definition \ref{238}.
Assume that $K_n$ converges to $K\subset\cl\Omega$ in the Hausdorff metric and that $\cL^2(K)=0$. 
Let $\C$ satisfy \eqref{904}, let $w_n\to w$ strongly in $W^{1,2}(\Omega;\R2)$, and let $u_n$ and $u$ be the solutions to the minimum problems 
\begin{equation}\label{332a}
\min\{\cE_{\C}(v,K_n):v\in\cA_{LD^2}(w_n,K_n)\} \qquad\text{and}\qquad\min\{\cE_{\C}(v,K):v\in\cA_{LD^2}(w,K)\}.
\end{equation}
Then $Eu_n1_{\Omega\setminus K_n}\to Eu1_{\Omega\setminus K}$ strongly in $L^2(\Omega;\M{2\times2}_\sym)$.
\end{theorem}
\begin{proof}
Arguing as in the first part of the proof of Theorem \ref{DMT5.1}, it is easy to see that there exists $u^*\in\cA_{LD^2}(w,K)$ such that $Eu_n1_{\Omega\setminus K_n}\wto Eu^*1_{\Omega\setminus K}$ up to a subsequence.
To prove that $u^*$ is a solution of the second minimum problem in \eqref{332a}, we now show that
\begin{equation}\label{906}
\cE_\C(u^*,K)\leq\cE_\C(v,K),\qquad\text{for every $v\in\cA_{LD^2}(w,K)$}.
\end{equation}
As in the proof of Theorem \ref{DMT5.1} we now approximate $v$ by a sequence $v^k$ of slightly more regular functions.
More precisely, using the approximation result \cite[Theorem 1]{Chambolle2003}, we find a sequence $v^k\in W^{1,2}(\Omega\setminus K;\R2)$ such that $Ev^k\to Ev$ strongly in $L^2(\Omega\setminus K;\M{2\times2}_\sym)$.
By applying Lemma \ref{nd100} component-wise, each $v^k$ is now approximated by a sequence $v_n^k\in W^{1,2}(\Omega\setminus K_n;\R2)$ such that $v_n^k=w_n$ on $\partial_D\Omega\setminus K_n$ and $Ev_n^k1_{\Omega\setminus K_n}\to Ev^k1_{\Omega\setminus K}$ strongly in $L^2(\Omega;\M{2\times2}_\sym)$.
We now continue as in the second part of the proof of Theorem \ref{DMT5.1} to prove \eqref{906}.
\end{proof}

\section{Appendix}
In this Appendix we prove the following result.
\begin{proposition}\label{a111}
Let $U\subset\R2$ be a bounded open set with Lipschitz boundary.
Then $LD^2(U)=W^{1,2}(U;\R2)$.
\end{proposition}
\begin{proof}
Let $u\in LD^2(U)$.
It is enough to prove that for all $x_0\in\partial U$ there exist an open rectangle $Q$ centered at $x_0$ such that $u\in W^{1,2}(Q\cap U;\R2)$.
It is not restrictive to assume that $x_0=0$. 
There exist an orthogonal coordinate system and two open intervals $I_a:=(-a,a)$ and $I_b:=(-b,b)$ such that $(I_a{\times}I_b)\cap U$ is the subgraph of a Lipschitz function $\psi\colon I_a\to I_{b/2}$.
For $0<\eta<\min\{a/4,b/4\}$ we define $u^\eta(x_1,x_2):=u(x_1,x_2-\eta)$ for $x_1\in I_a$ and $-b+\eta<x_2<\psi(x_1)+\eta$.
Let $Q:=I_{3a/4}{\times}I_{3b/4}$ and let $D:=I_{3a/4}{\times}(-\tfrac34b,-\tfrac23b)\Subset U$.
For every $0<\eps<\eta$ let $\rho_\eps$ be a convolution kernel with support contained in the ball centered at $0$ and radius $\eps$.
Finally we define $u_\eps^\eta:=u^\eta*\rho_\eps$, which is well defined in $Q\cap U$.
By Korn's inequality, there exists a positive constant $c$ depending only on $Q\cap U$ and $D$, but not on $\eps$ and $\eta$, such that
\begin{equation*}
\norma{u_\eps^\eta}_{L^2(Q\cap U;\R2)}+\norma{\nabla u_\eps^\eta}_{L^2(Q\cap U;\M{2{\times}2})}\leq c\norma{Eu_\eps^\eta}_{L^2(Q\cap U;\M{2{\times}2}_\sym)}+c\norma{u_\eps^\eta}_{L^2(D;\R2)}.
\end{equation*}
The usual formulation of Korn's inequality (see, e.g., \cite[Theorem 6.3-3]{Ciarlet}) has $Q\cap U$ instead of $D$ in the last term.
Our formulation can be easily obtained from the compact embedding of $W^{1,2}(Q\cap U;\R2)$ into $L^2(Q\cap U;\R2)$, arguing by contradiction as in the proof of the Poincar\'e inequality.

Using the fact that $Eu\in L^2(Q\cap U;\M{2{\times}2}_\sym)$ and $u\in L^2(D;\R2)$, we can pass to the limit first when $\eps\to0$ and then when $\eta\to0$ and we obtain that $u\in L^2(Q\cap U;\R2)$ and $\nabla u\in L^2(Q\cap U;\M{2{\times}2})$, hence $u\in W^{1,2}(Q\cap U;\R2)$.
\end{proof}

\noindent {\bf Acknowledgements.} 
This work was partially supported by the ERC Advanced Grant QuaDynEvoPro, \emph{``Quasistatic and Dynamic Evolution Problems in Plasticity and Fracture''}, grant agreement no.\@ 290888.
The authors are members of the Gruppo Nazionale per l'Analisi Matematica, la Probabilità e le loro Applicazioni (GNAMPA) of the Istituto Nazionale di Alta Matematica (INdAM).
M.M.\@ is a member of the Progetto di Ricerca GNAMPA-INdAM 2015 \emph{``Fenomeni critici nella meccanica dei materiali: un approccio variazionale''}. 

\end{document}